\documentclass{article}
\usepackage{amsthm}
\usepackage{amsmath}
\usepackage{amssymb}
\usepackage{graphicx}
\usepackage{varwidth}
\usepackage{algpseudocode}
\usepackage{algorithm}
\usepackage{color}
\usepackage{tikz}
\usepackage{graphicx}
\usetikzlibrary{arrows}

\newtheorem{thm}{Theorem}[section]

\newtheorem{lem}[thm]{Lemma}

\newtheorem{conj}[thm]{Conjecture}

\theoremstyle{definition}
\newtheorem{defn}[thm]{Definition}

\theoremstyle{remark}

\newcommand{\A}{\mathcal{A}}

\begin{document}

\title{Locating Patterns in the De Bruijn Torus}
\author{Victoria Horan \thanks{\texttt{victoria.horan.1@us.af.mil}} \\  Air Force Research Laboratory \\ Information Directorate \\ \\ Brett Stevens \thanks{\texttt{brett@math.carleton.ca}} \\ School of Mathematics and Statistics \\ Carleton University}%


\date{\today}%

\maketitle

\let\thefootnote\relax\footnote{Approved for public release; distribution unlimited:  88ABW-2015-2248.}

\begin{abstract}
    The de Bruijn torus (or grid) problem looks to find an $n$-by-$m$ binary matrix in which every possible $j$-by-$k$ submatrix appears exactly once.  The existence and construction of these binary matrices was determined in the 70's, with generalizations to $d$-ary matrices in the 80's and 90's.  However, these constructions lacked efficient decoding methods, leading to new constructions in the early 2000's.  The new constructions develop cross-shaped patterns (rather than rectangular), and rely on a concept known as a half de Bruijn sequence.  In this paper, we further advance this construction beyond cross-shape patterns.  Furthermore, we show results for universal cycle grids, based off of the one-dimensional universal cycles introduced by Chung, Diaconis, and Graham, in the 90's.  These grids have many applications such as robotic vision, location detection, and projective touch-screen displays.
\end{abstract}

\section{Introduction}

During the Workshop on Generalizations of de Bruijn Cycles and Gray Codes at the Banff International Research Station in December 2004, Ron Graham proposed Problem 480:  De Bruijn Tori \cite{ResearchProbs}.  In short, a de Bruijn torus is an $r \times v$ $d$-ary array embedded on a torus in which every possible $n \times m$ array appears exactly once.  These types of tori or grids are extremely useful in many applications, such as robotic vision \cite{MagicalMath} and projected touch screens \cite{TouchScreen}.

While much work has been done on the existence of these tori (see \cite{HI1995}, for example), current methods require more efficient decoding algorithms.  To cope with these difficulties, new constructions were developed that produced alternative window sizes and shapes instead of rectangular subarrays.  For example, in \cite{TEtzion} a construction with a cross-shaped window was developed with a decoding algorithm that allowed for a far more efficient location discovery method.  In this paper, we expand on this work and show a wider range of window options.  With our new construction, we reduce the brute-force complexity of $\mathcal{O}(d^{2n})$ for a $d$-ary grid with window size $n \times n$ down to $\mathcal{O}(d^n)$.  Coupled with recent work on the infamous problem of ranking de Bruijn sequences, by choosing appropriate sequences to base our grid off of, this complexity may be reduced even further down to $\mathcal{O}(n^3)$ (using results from \cite{cubic}) or $\mathcal{O}(n^2)$ (using results from \cite{quadratic}).

In Section \ref{background} we provide the necessary definitions and relevant background.  Section \ref{torus} develops new results on the de Bruijn torus problem, while Section \ref{ucycles} generalizes this problem from de Bruijn sequences to universal cycles.  Finally, Section \ref{FutureWork} explores possible future research directions in this area.

\section{Background and Definitions}\label{background}

For the unfamiliar reader, we provide the following definitions and brief history of this research problem.

\begin{defn}
    A $d$-ary \textbf{de Bruijn sequence} of order $n$ is a string $D=x_0x_1x_2 \ldots x_{d^n+n-2}$ such that every $n$-tuple over the alphabet $[d]= \{1,2,\ldots , d\}$ appears exactly once as $x_i x_{i+1} \ldots x_{i+n-1}$.
\end{defn}

De Bruijn sequences are often cyclic, meaning that the last letter is adjacent to the first, and $n$-tuples are allowed to `wrap around' from end to beginning.  In this case, the sequences are often called de Bruijn cycles.

\begin{defn}
     A \textbf{de Bruijn array} (or $(r,v;n,m)_d$-array) is an $r \times v$ $d$-ary array in which every window of size $n \times m$ appears exactly once.  A \textbf{de Bruijn torus} is a de Bruijn array in which the last row is adjacent to the first row, and similarly the last column is adjacent to the first column.
\end{defn}

Note that any de Bruijn torus can be easily converted to a de Bruijn array, but not necessarily vice versa.

\begin{defn}
    A \textbf{pseudo-random sequence} is a $d$-ary de Bruijn sequence of order $n$ that is missing the term $0^n$ and is developed from a linear feedback shift register based on a primitive polynomial of degree $n$.
\end{defn}

\begin{defn}
    \cite{PNSeq1988}  A \textbf{pseudo-random array} is an array in which every \textit{nonzero} window appears exactly once, and we will call it an $(r,v;n,m)_d$-PNarray.
\end{defn}

A generalization of de Bruijn sequences that allows for combinatorial objects other than simply $d$-ary words is a universal string, introduced in \cite{CDG}.

\begin{defn}
    A \textbf{universal string}, or ustring, over a set of combinatorial objects $\mathcal{C}$, each of order $n$, is a sequence $U = a_0 a_1 \ldots a_{m-1}$ in which each object is represented exactly once as a consecutive substring $a_{i+1}a_{i+2} \ldots a_n$.
\end{defn}

As with de Bruijn sequences, if a universal string is cyclic and wraps around (i.e. subscript addition is modulo $m$), we call it a \textbf{universal cycle} or \textbf{ucycle}.  With respect to our previous definitions, a de Bruijn cycle is simply a universal cycle with $\mathcal{C}$ equal to the set of all $d$-ary strings of length $n$.

Past work in this area concerns pseudo-random and de Bruijn arrays and tori with rectangular windows.  These results are summarized succinctly below.

\begin{thm}
    There exists an $(r,v;n,m)_d$-array whenever:
    \begin{itemize}
        \item  \emph{\cite{Etzion1988}}  $d=2,r=2^\ell,v=2^{nm-\ell}$, where either $n<2^\ell \leq 2^n$, or $m \neq 2$ if $\ell=n$.
        \item  \emph{\cite{HI1995}}  If $d$ has prime decomposition $d = \Pi p_{\ell}^{\alpha_\ell}$ and we define $q = d \Pi p_{\ell}^{\lfloor \log_{p_\ell} n \rfloor}$, then $r=q$ and $v = d^{nm/q}$.
    \end{itemize}
\end{thm}

\begin{thm}
    There exists an $(r,v;n,m)_d$-PNarray whenever:
    \begin{itemize}
        \item  \emph{\cite{PNSeq1988}}  $d=2,r=2^n-1, v=2^{nm-1}/r$, and gcd$(r,v)=1$.
    \end{itemize}
\end{thm}

These results mostly concern the problem of existence and do not explicitly discuss algorithms or complexity for locating specific structures.  To deal with the decoding problem, several others considered alternative window shapes.  This method will be explored and generalized in the following section.  

Additionally, it is conjectured in \cite{HMP1996} that the following necessary conditions are also sufficient.

\begin{conj}
    \emph{\cite{HMP1996}}  There exists an $(r,v;n,m)_d$-array whenever:
    \begin{enumerate}
            \item  $rv = d^{nm}$ (only if it is a torus),
            \item  $r>n$ or $r=n=1$, and
            \item  $v>m$ or $v=m=1$.
        \end{enumerate}
\end{conj}

\section{The De Bruijn Torus Problem}\label{torus}

In order to use the structures defined in Section 2 for applications like robot location, we must be able to determine efficiently where a particular subsequence occurs.  A partial solution to this problem is presented in \cite{TEtzion, Scheinerman}, in which the authors consider cross-shaped block patterns rather than a rectangular subarray.  These cross-shaped patterns have a set of consecutive blocks horizontally and a set of consecutive blocks vertically, and these sets overlap in exactly one block.  That is, if we have $n$ horizontal blocks and $k$ vertical blocks, the cross contains a total of $n+k-1$ blocks.  In this section, we will present this approach and expand upon it to allow for a more generalized set of pattern rules.  In the next section, we will also consider universal cycles over other types of combinatorial objects instead of solely using de Bruijn sequences over $d$-ary strings.

We begin with some more definitions.

\begin{defn}
    Let $\mathcal{G}$ be a group of order $d$.  Define an equivalence relation on $\mathcal{G}^n$ (or $d$-ary $n$-tuples) as follows.  We set $x \equiv y$ if and only if $x-y = c \cdot (1,1, \ldots , 1)$ for some $c \in \mathcal{G}$.  Then we define the \textbf{quotient de Bruijn string}, $D$, to be a string of length $d^{n-1}$ such that every equivalence class has exactly one representative appear exactly once in $D$.
\end{defn}

For example, a quotient string for binary de Bruijn sequences is known as the half de Bruijn sequence.  This string utilizes the equivalence relation given by $x \sim y$ if and only if $$(x_1+y_1, x_2+y_2, \ldots , x_n+y_n)  \in \{0^n, 1^n\},$$ with the bitwise addition performed modulo 2.  For a good discussion on the computational complexity of constructing these half de Bruijn sequences (otherwise known as complement-free de Bruijn sequences), see \cite{Scheinerman}.  For quotient de Bruijn strings over larger alphabets, we provide the following result.

\begin{thm}
    Let $\mathcal{G}$ be a group of order $d$ with operation `$+$'.  Let $\mathcal{A}=a_0a_1 \ldots a_{d^{n-1}-1}$ be a $d$-ary de Bruijn sequence for $(n-1)$-tuples.  Define the string $\mathcal{D} = d_0d_1 \ldots d_{d^{n-1}-1}$ such that $d_{i+1}=a_i+d_i$.  Then $\mathcal{D}$ is a quotient de Bruijn string.
\end{thm}
\begin{proof}
    First, we note that the equivalence classes for $d$-ary strings have size $d$, and the union of all equivalence classes for $d$-ary strings of length $n$ has size $d^n$.  Thus since $\mathcal{D}$ contains $d^{n-1}$ different strings, we need only show that no two representatives from the same equivalence class appear in $\mathcal{D}$.

    To show this, we proceed by contradiction and suppose that strings $(x_1+k)(x_2+k) \cdots (x_n+k)$ and $(x_1+\ell)(x_2+\ell) \cdots (x_n+\ell)$ appear starting in positions $p$ and $q$ in $\mathcal{D}$, respectively.  Then we have the following equalities.

    $$\begin{array}{rclcl}
        d_p & = & x_1+k & = & a_{p-1}+d_{p-1} \\
        d_{p+1} & = & x_2+k & = & a_p + d_p \\
        & \vdots & & \vdots & \\
        d_{p+n-1} & = & x_n+k & = & a_{p+n-2} + d_{p+n-2}
    \end{array}$$

    and

    $$\begin{array}{rclcl}
        d_q & = & x_1+\ell & = & a_{q-1}+d_{q-1} \\
        d_{q+1} & = & x_2+\ell & = & a_q + d_q \\
        & \vdots & & \vdots & \\
        d_{q+n-1} & = & x_n+\ell & = & a_{q+n-2} + d_{q+n-2}
    \end{array}$$

    Note that these simplify down to the following.

    $$\begin{array}{rclcl}
        a_p & = & x_2-x_1 & = & a_q \\
        a_{p+1} & = & x_3-x_2 & = & a_{q+1} \\
        & \vdots & & \vdots & \\
        a_{p+n-2} & = & x_n - x_{n-1} & = & a_{q+n-2}
    \end{array}$$

    Thus if both $(x_1+k)(x_2+k) \cdots (x_n+k)$ and $(x_1+\ell)(x_2+\ell) \cdots (x_n+\ell)$ appear in $\mathcal{D}$, then the $(n-1)$-tuple $(x_2-x_1)(x_3-x_2) \cdots (x_n-x_{n-1})$ appears twice in $\mathcal{A}$, which contradicts that $\mathcal{A}$ is a de Bruijn sequence.
\end{proof}

We will use these quotient strings in the construction of a torus, as defined below.

\begin{defn}
    Let $\mathcal{C}$ be a universal cycle for a set of objects over an alphabet of size $d$, and let $\mathcal{Q}$ be a quotient string for a (possibly different) set of objects over the same alphabet of size $d$.  The $\mathcal{Q} \times \mathcal{C}$ \textbf{grid} (or \textbf{torus}) is the rectangular grid with rows labeled from $1$ to $|\mathcal{Q}|$ and columns labeled from $1$ to $|\mathcal{C}|$, and where the entry in row $i$ and column $j$ is $\mathcal{C}_j +\mathcal{Q}_i$ where `$+$' denotes the binary operation for a group $\mathcal{G}$ of order $d$ utilizing symbols from our common alphabet.
\end{defn}

For example, when considering binary de Bruijn sequences, there will be $2^n$ columns and $2^{k-1}$ rows, and addition will be in $\mathbb{Z}/2\mathbb{Z}$.  When considering patterns in the $k \times n$ torus, we will use the following definition of block patterns, which will be possible window patterns for our grids.

\begin{defn}
    A $k \times n$ \textbf{block pattern} in a $\mathcal{Q} \times \mathcal{C}$ torus is a selection of entries (or \textbf{blocks}) within a subarray of dimension $k \times n$.  Here and in other literature this is often also referred to as a \textbf{window}.  When considering not just the general shape of a block pattern but a specific instance of a block pattern in a grid, we say that the block pattern is \textbf{filled}.
\end{defn}

For example, one possible block pattern in a $4 \times 4$ region is given in Figure \ref{BP}.  In this example, the black blocks represent our block pattern.  Several of the latest de Bruijn torus results utilize cross-shaped patterns.  These are patterns that contain $n$ consecutive blocks in one row, $k$ consecutive blocks in one column, and one block in common for a total of $n+k-1$ blocks.  For a cross-shaped block pattern, the following theorem gives a nice result.

\begin{thm}\label{EtzionCross}
    \emph{\cite{TEtzion}}  Let $\mathcal{D}$ be a binary de Bruijn sequence of order $n$ and $\mathcal{Q}$ be a half de Bruijn sequence of order $k$.  Produce the $\mathcal{Q} \times \mathcal{D}$ grid using binary addition.  Fix some block pattern $\mathcal{B}$ that is a $k \times n$ cross (i.e. $n$ entries in one row, $k$ entries in one column, overlapping in exactly one block).  Then in the constructed grid we can find every binary-filled block pattern $\mathcal{B}$ exactly once.
\end{thm}

We will generalize Theorem \ref{EtzionCross} so that the block pattern is not required to be cross-shaped, but is required to have certain horizontal and vertical projections, as well as a few additional restrictions on the connection graph for the given block pattern.  From a given block pattern, we want to create a graph in which blocks correspond to nodes, edges correspond to nearest neighbors in the north/south/east/west directions.

\begin{defn}
    The \textbf{connection graph} for a $k \times n$ block pattern is created as follows.  For each block in the block pattern, draw a node.  We then draw undirected edges corresponding to nearest neighbors in each direction (north, south, east, and west).
\end{defn}

Figure \ref{BP} shows a block pattern and its corresponding connection graph.  In order to consider finding patterns that will satisfy de Bruijn-type properties, we will need a few more definitions.

\begin{figure}
\begin{center}
\begin{tikzpicture}[scale= 1 ]

\node[fill=black,inner sep=0.5cm,outer sep=0pt,anchor=south west] at (0,3) {};
\node[fill=black,inner sep=0.5cm,outer sep=0pt,anchor=south west] at (1,3) {};
\node[fill=black,inner sep=0.5cm,outer sep=0pt,anchor=south west] at (2,3) {};
\node[fill=black,inner sep=0.5cm,outer sep=0pt,anchor=south west] at (1,2) {};
\node[fill=black,inner sep=0.5cm,outer sep=0pt,anchor=south west] at (3,2) {};
\node[fill=black,inner sep=0.5cm,outer sep=0pt,anchor=south west] at (2,0) {};

  \draw[step=1,help lines] (0,0) grid (4,4);

\end{tikzpicture}
\hspace{10mm}
\begin{tikzpicture}[-,>=stealth',auto,node distance=2cm,
  thick,main node/.style={circle,draw,font=\sffamily\bfseries}]

  \node[main node] (1) at (0,3)             {};
  \node[main node] (2) at (1,3)             {};
  \node[main node] (3) at (2,3)             {};
  \node[main node] (4) at (1,2)             {};
  \node[main node] (5) at (3,2)             {};
  \node[main node] (6) at (2,0)             {};

  \path[every node/.style={font=\sffamily\footnotesize}]
    (1) edge node [left]        {} (2)
    (2) edge node [left]        {} (3)
    (2) edge node [left]        {} (4)
    (3) edge node [left]        {} (6)
    (4) edge node [right]       {} (5)
    ;

\end{tikzpicture}
\end{center}
\caption{A possible block pattern in a $4 \times 4$ region and its corresponding connection graph.}\label{BP}
\end{figure}
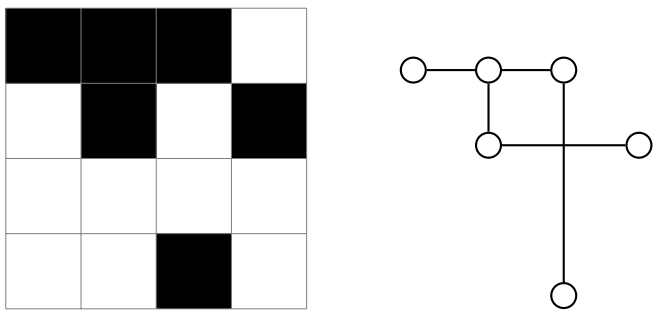

\begin{defn}
    The \textbf{projection} of an $s \times t$ block pattern on the horizontal is given as a binary $t$-sequence in which entry $i$ equals $1$ if and only if there is at least one block used in column $i$ for $i \in [t]$.  Similarly, the projection on the vertical is a binary $s$-sequence in which we consider the rows instead of columns.  When considering a filled block pattern, we replace the 0's with `-' and the 1's with the appropriate row or column entry.  For example, in Figure \ref{DBGEx}, the row projection is $(a, -, -, \overline{a}, -, -, \overline{a})$ and the column projection is $(a, \overline{a}, -, -, \overline{a})$.
\end{defn}

Now we may consider block patterns with non-consecutive projections, rather than simply the standard de Bruijn sequence.  One variation used will be combs and their corresponding sequences.

\begin{defn}
    A \textbf{comb} of order $n$ is a binary sequence $x_0x_1 \ldots x_{s}$ such that it contains exactly $n$ ones.  We will represent the comb as a sequence of indices for the non-zero entries, i.e. the comb $01011$ corresponds to sequence $[1,3,4]$.
\end{defn}

We think of combs as a modified window.  For standard de Bruijn sequences, we use the comb $11 \cdots 1$.  However, an alternative comb for binary de Bruijn sequences for strings of length 3 is $[0,2,4]$.  Following the notation of \cite{CombsGraham}, this corresponds to the window \texttt{O*O*O}.  One de Bruijn sequence for this comb is the following.  $$11010100$$  This window produces the following sequence of binary triples:  $100,111,000,110$, $001,101,010,011$.  De Bruijn sequences for given comb patterns have been studied in the literature.  For example, see \cite{CombsGraham, Cooper}.

Note that our projection must match the structure of the objects used to create the torus.  In terms of our example in Figure \ref{DBGEx}, this means that we must have a binary de Bruijn sequence for strings of length 3 with comb $[0,1,4]$ and a binary quotient string for strings of length 3 with comb $[0,3,6]$.

The following theorem is our main result, and serves to generalize Theorem \ref{EtzionCross} from \cite{TEtzion}.  Instead of using simple $n \times k$ cross shaped block patterns, our result allows for a wide variety of block patterns, from block patterns containing $n+k-1$ blocks in an $n \times k$ subgrid to block patterns of $n+k-1$ block contained in a much larger subgrid with a variety of gaps in their projections (i.e. the projections are combs).

\begin{thm}\label{tree}
    Let $\mathcal{S}$ be a de Bruijn sequence for strings of length $n$ over an alphabet of size $d$ with a given comb pattern $P_1$.  Let $\mathcal{Q}$ be a quotient string over the same alphabet of size $d$ with comb pattern $P_2$ for strings of length $k$.  Place $n+k-1$ blocks on the grid in a block pattern $\mathcal{B}$ so the connection graph for $\mathcal{B}$ is a tree, and so that the horizontal projection of $\mathcal{B}$ corresponds to $P_1$ and the vertical projection of $\mathcal{B}$ corresponds to $P_2$.  Then the grid $\mathcal{T}$ produced by $\mathcal{Q} \times \mathcal{S}$ and addition modulo $d$ contains every possible $n+k-1$ binary combination in block pattern $\mathcal{B}$ exactly once.
\end{thm}
\begin{proof}
    The total number of possible $(n+k-1)$ combinations is $d^{n+k-1}$.  We know that $\mathcal{S}$ has length $d^n$ and $\mathcal{Q}$ has length $d^{k-1}$, so our grid has size $d^{n+k-1}$.  Thus if we show that we can find any $(n+k-1)$-string in our block pattern, we are done.

    Since our connection graph is a tree, label our blocks $B_1, B_2, \ldots , B_{n+k-1}$ so that $B_1$ is a root and every child appears after its parent.  This can be done using a breadth-first search method.  Fill in the blocks in this order arbitrarily with bits to make an arbitrary $(n+k-1)$-binary string $x_1x_2 \ldots x_{n+k-1}$.

    If our string is to appear in the grid, we can determine all row/column labels as follows.  Suppose $B_1$ appears in row labelled $a$.  Then $B_1$ must be in column $a+x_1$ modulo $d$.  As we work through the list in this fashion, each block $B_i$ must have exactly one of the row or column labels assigned as we reach it in the list.  Call the known label $\alpha$.  Then the remaining label to be determined will always be $\alpha + x_i$.  Thus we can determine the projection sequence $r_1r_2 \ldots r_k$ for the row labels and $c_1c_2 \ldots c_n$ for the column labels, all solely based on our original label $a$.  Since $\mathcal{S}$ is a de Bruijn sequence, we know that $(c_1+x)(c_2+x) \ldots (c_n+x)$ must appear in $\mathcal{S}$ for all choices of $x$.  Since $\mathcal{Q}$ is a quotient string, exactly one of $(r_1+x)(r_2+x) \ldots (r_k+x)$ must appear.  If $(r_1+x)(r_2+x) \ldots (r_k+x)$ appears, our block pattern for this $(n+k-1)$ combination appears in $[(r_1+x)(r_2+x) \ldots (r_k+x)] \times [(c_1-x)(c_2-x) \ldots (c_n-x)]$.  
\end{proof}

We now give the following example with $n=3$ and $k=3$.  Suppose that our alphabet is $\A = \{0,1\}$ which will correspond to white/black, respectively.  Let our comb patterns be the following:  $P_1 = [0,1,4]$ and $P_2 = [0,3,6]$.  Let $\mathcal{S} = 11100100$ and $\mathcal{Q} = 1110$.  We construct the following torus $\mathcal{T}$ from $\mathcal{Q} \times \mathcal{S}$.
$$\begin{array}{cccccccc}
    0 & 0 & 0 & 1 & 1 & 0 & 1 & 1 \\
    0 & 0 & 0 & 1 & 1 & 0 & 1 & 1 \\
    0 & 0 & 0 & 1 & 1 & 0 & 1 & 1 \\
    1 & 1 & 1 & 0 & 0 & 1 & 0 & 0
\end{array}$$

We will construct our grid so as not to require wrap-around, and paint the tiles black/white according to our instruction.  This produces the grid shown in Figure \ref{DBGEx}.  Now suppose that in this grid we are looking for the specific block pattern shown next to the grid in the figure.  We also include the connection graph in the figure.  If our top left red block appears in column $a$, then we get the column/row values shown on the relevant rows and columns.

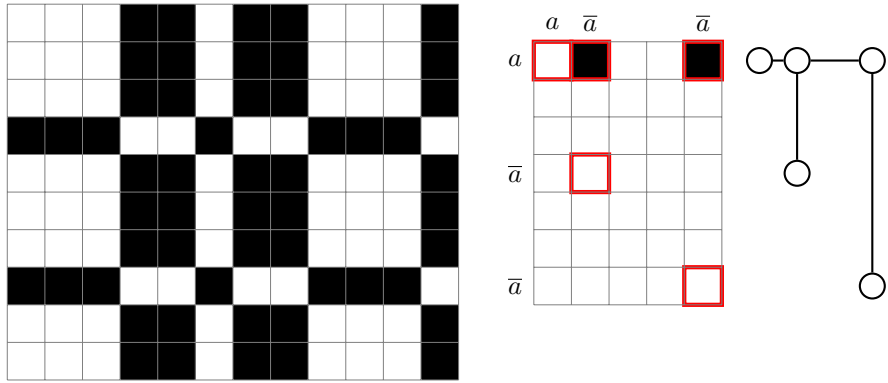
\begin{figure}
\begin{center}
\begin{tikzpicture}[-,>=stealth',auto,node distance=2cm,
  thick,main node/.style={circle,draw,font=\sffamily\bfseries}, scale= 0.5 ]

\node[fill=black,inner sep=0.25cm,outer sep=0pt,anchor=south west] at (0,2) {};
\node[fill=black,inner sep=0.25cm,outer sep=0pt,anchor=south west] at (0,6) {};
\node[fill=black,inner sep=0.25cm,outer sep=0pt,anchor=south west] at (1,2) {};
\node[fill=black,inner sep=0.25cm,outer sep=0pt,anchor=south west] at (1,6) {};
\node[fill=black,inner sep=0.25cm,outer sep=0pt,anchor=south west] at (2,2) {};
\node[fill=black,inner sep=0.25cm,outer sep=0pt,anchor=south west] at (2,6) {};
\node[fill=black,inner sep=0.25cm,outer sep=0pt,anchor=south west] at (5,2) {};
\node[fill=black,inner sep=0.25cm,outer sep=0pt,anchor=south west] at (5,6) {};
\node[fill=black,inner sep=0.25cm,outer sep=0pt,anchor=south west] at (8,2) {};
\node[fill=black,inner sep=0.25cm,outer sep=0pt,anchor=south west] at (8,6) {};
\node[fill=black,inner sep=0.25cm,outer sep=0pt,anchor=south west] at (9,2) {};
\node[fill=black,inner sep=0.25cm,outer sep=0pt,anchor=south west] at (9,6) {};
\node[fill=black,inner sep=0.25cm,outer sep=0pt,anchor=south west] at (10,2) {};
\node[fill=black,inner sep=0.25cm,outer sep=0pt,anchor=south west] at (10,6) {};

\node[fill=black,inner sep=0.25cm,outer sep=0pt,anchor=south west] at (3,0) {};
\node[fill=black,inner sep=0.25cm,outer sep=0pt,anchor=south west] at (3,1) {};
\node[fill=black,inner sep=0.25cm,outer sep=0pt,anchor=south west] at (3,3) {};
\node[fill=black,inner sep=0.25cm,outer sep=0pt,anchor=south west] at (3,4) {};
\node[fill=black,inner sep=0.25cm,outer sep=0pt,anchor=south west] at (3,5) {};
\node[fill=black,inner sep=0.25cm,outer sep=0pt,anchor=south west] at (3,7) {};
\node[fill=black,inner sep=0.25cm,outer sep=0pt,anchor=south west] at (3,8) {};
\node[fill=black,inner sep=0.25cm,outer sep=0pt,anchor=south west] at (3,9) {};
\node[fill=black,inner sep=0.25cm,outer sep=0pt,anchor=south west] at (4,0) {};
\node[fill=black,inner sep=0.25cm,outer sep=0pt,anchor=south west] at (4,1) {};
\node[fill=black,inner sep=0.25cm,outer sep=0pt,anchor=south west] at (4,3) {};
\node[fill=black,inner sep=0.25cm,outer sep=0pt,anchor=south west] at (4,4) {};
\node[fill=black,inner sep=0.25cm,outer sep=0pt,anchor=south west] at (4,5) {};
\node[fill=black,inner sep=0.25cm,outer sep=0pt,anchor=south west] at (4,7) {};
\node[fill=black,inner sep=0.25cm,outer sep=0pt,anchor=south west] at (4,8) {};
\node[fill=black,inner sep=0.25cm,outer sep=0pt,anchor=south west] at (4,9) {};
\node[fill=black,inner sep=0.25cm,outer sep=0pt,anchor=south west] at (6,0) {};
\node[fill=black,inner sep=0.25cm,outer sep=0pt,anchor=south west] at (6,1) {};
\node[fill=black,inner sep=0.25cm,outer sep=0pt,anchor=south west] at (6,3) {};
\node[fill=black,inner sep=0.25cm,outer sep=0pt,anchor=south west] at (6,4) {};
\node[fill=black,inner sep=0.25cm,outer sep=0pt,anchor=south west] at (6,5) {};
\node[fill=black,inner sep=0.25cm,outer sep=0pt,anchor=south west] at (6,7) {};
\node[fill=black,inner sep=0.25cm,outer sep=0pt,anchor=south west] at (6,8) {};
\node[fill=black,inner sep=0.25cm,outer sep=0pt,anchor=south west] at (6,9) {};
\node[fill=black,inner sep=0.25cm,outer sep=0pt,anchor=south west] at (7,0) {};
\node[fill=black,inner sep=0.25cm,outer sep=0pt,anchor=south west] at (7,1) {};
\node[fill=black,inner sep=0.25cm,outer sep=0pt,anchor=south west] at (7,3) {};
\node[fill=black,inner sep=0.25cm,outer sep=0pt,anchor=south west] at (7,4) {};
\node[fill=black,inner sep=0.25cm,outer sep=0pt,anchor=south west] at (7,5) {};
\node[fill=black,inner sep=0.25cm,outer sep=0pt,anchor=south west] at (7,7) {};
\node[fill=black,inner sep=0.25cm,outer sep=0pt,anchor=south west] at (7,8) {};
\node[fill=black,inner sep=0.25cm,outer sep=0pt,anchor=south west] at (7,9) {};
\node[fill=black,inner sep=0.25cm,outer sep=0pt,anchor=south west] at (11,0) {};
\node[fill=black,inner sep=0.25cm,outer sep=0pt,anchor=south west] at (11,1) {};
\node[fill=black,inner sep=0.25cm,outer sep=0pt,anchor=south west] at (11,3) {};
\node[fill=black,inner sep=0.25cm,outer sep=0pt,anchor=south west] at (11,4) {};
\node[fill=black,inner sep=0.25cm,outer sep=0pt,anchor=south west] at (11,5) {};
\node[fill=black,inner sep=0.25cm,outer sep=0pt,anchor=south west] at (11,7) {};
\node[fill=black,inner sep=0.25cm,outer sep=0pt,anchor=south west] at (11,8) {};
\node[fill=black,inner sep=0.25cm,outer sep=0pt,anchor=south west] at (11,9) {};

  \draw[step=1,help lines] (0,0) grid (12,10);

\node[draw=red,ultra thick,inner sep=0.25cm,outer sep=0pt,anchor=south west] at (14,8) {};
\node[draw=red,fill=black,ultra thick,inner sep=0.25cm,outer sep=0pt,anchor=south west] at (15,8) {};
\node[draw=red,fill=black,ultra thick,inner sep=0.25cm,outer sep=0pt,anchor=south west] at (18,8) {};
\node[draw=red,ultra thick,inner sep=0.25cm,outer sep=0pt,anchor=south west] at (15,5) {};
\node[draw=red,ultra thick,inner sep=0.25cm,outer sep=0pt,anchor=south west] at (18,2) {};

  \draw[step=1,help lines] (14,2) grid (19,9);

  \node[] at (14.5,9.5) {$a$};
\node[] at (13.5,8.5) {$a$};
\node[] at (13.5,5.5) {$\overline{a}$};
\node[] at (13.5,2.5) {$\overline{a}$};
\node[] at (15.5,9.5) {$\overline{a}$};
\node[] at (18.5,9.5) {$\overline{a}$};

  \node[main node] (1) at (20,8.5)             {};
  \node[main node] (2) at (21,8.5)             {};
  \node[main node] (3) at (23,8.5)             {};
  \node[main node] (4) at (21,5.5)             {};
  \node[main node] (5) at (23,2.5)             {};

  \path[every node/.style={font=\sffamily\footnotesize}]
    (1) edge node [left]        {} (2)
    (2) edge node [left]        {} (3)
    (2) edge node [left]        {} (4)
    (3) edge node [right]       {} (5)
    ;

\end{tikzpicture}
\end{center}

\caption{De Bruijn grid and pattern to be located with corresponding connection graph.}\label{DBGEx}
\end{figure}

First, we look at the vertical projection given by $1001001$.  Where does a sequence of type $(a, -, -, \overline{a}, -, -, \overline{a})$ appear?  In position/row 4 as $(0, -, -, 1, -, -, 1)$.  This sets our value $a = 0$.  Next, the horizontal projection given by $11001$.  Where does the sequence $(0, 1, -, -, 1)$ appear?  In position/column 5.  This gives us the position of our robot on the grid, shown in Figure \ref{DBGExS}.

\begin{figure}
\begin{center}
\begin{tikzpicture}[scale= 0.5 ]

\node[fill=black,inner sep=0.25cm,outer sep=0pt,anchor=south west] at (0,2) {};
\node[fill=black,inner sep=0.25cm,outer sep=0pt,anchor=south west] at (0,6) {};
\node[fill=black,inner sep=0.25cm,outer sep=0pt,anchor=south west] at (1,2) {};
\node[fill=black,inner sep=0.25cm,outer sep=0pt,anchor=south west] at (1,6) {};
\node[fill=black,inner sep=0.25cm,outer sep=0pt,anchor=south west] at (2,2) {};
\node[fill=black,inner sep=0.25cm,outer sep=0pt,anchor=south west] at (2,6) {};
\node[fill=black,inner sep=0.25cm,outer sep=0pt,anchor=south west] at (5,2) {};
\node[fill=black,inner sep=0.25cm,outer sep=0pt,anchor=south west] at (5,6) {};
\node[fill=black,inner sep=0.25cm,outer sep=0pt,anchor=south west] at (8,2) {};
\node[fill=black,inner sep=0.25cm,outer sep=0pt,anchor=south west] at (8,6) {};
\node[fill=black,inner sep=0.25cm,outer sep=0pt,anchor=south west] at (9,2) {};
\node[fill=black,inner sep=0.25cm,outer sep=0pt,anchor=south west] at (9,6) {};
\node[fill=black,inner sep=0.25cm,outer sep=0pt,anchor=south west] at (10,2) {};
\node[fill=black,inner sep=0.25cm,outer sep=0pt,anchor=south west] at (10,6) {};

\node[fill=black,inner sep=0.25cm,outer sep=0pt,anchor=south west] at (3,0) {};
\node[fill=black,inner sep=0.25cm,outer sep=0pt,anchor=south west] at (3,1) {};
\node[fill=black,inner sep=0.25cm,outer sep=0pt,anchor=south west] at (3,3) {};
\node[fill=black,inner sep=0.25cm,outer sep=0pt,anchor=south west] at (3,4) {};
\node[fill=black,inner sep=0.25cm,outer sep=0pt,anchor=south west] at (3,5) {};
\node[fill=black,inner sep=0.25cm,outer sep=0pt,anchor=south west] at (3,7) {};
\node[fill=black,inner sep=0.25cm,outer sep=0pt,anchor=south west] at (3,8) {};
\node[fill=black,inner sep=0.25cm,outer sep=0pt,anchor=south west] at (3,9) {};
\node[fill=black,inner sep=0.25cm,outer sep=0pt,anchor=south west] at (4,0) {};
\node[fill=black,inner sep=0.25cm,outer sep=0pt,anchor=south west] at (4,1) {};
\node[fill=black,inner sep=0.25cm,outer sep=0pt,anchor=south west] at (4,3) {};
\node[fill=black,inner sep=0.25cm,outer sep=0pt,anchor=south west] at (4,4) {};
\node[fill=black,inner sep=0.25cm,outer sep=0pt,anchor=south west] at (4,5) {};
\node[fill=black,inner sep=0.25cm,outer sep=0pt,anchor=south west] at (4,7) {};
\node[fill=black,inner sep=0.25cm,outer sep=0pt,anchor=south west] at (4,8) {};
\node[fill=black,inner sep=0.25cm,outer sep=0pt,anchor=south west] at (4,9) {};
\node[fill=black,inner sep=0.25cm,outer sep=0pt,anchor=south west] at (6,0) {};
\node[fill=black,inner sep=0.25cm,outer sep=0pt,anchor=south west] at (6,1) {};
\node[fill=black,inner sep=0.25cm,outer sep=0pt,anchor=south west] at (6,3) {};
\node[fill=black,inner sep=0.25cm,outer sep=0pt,anchor=south west] at (6,4) {};
\node[fill=black,inner sep=0.25cm,outer sep=0pt,anchor=south west] at (6,5) {};
\node[fill=black,inner sep=0.25cm,outer sep=0pt,anchor=south west] at (6,7) {};
\node[fill=black,inner sep=0.25cm,outer sep=0pt,anchor=south west] at (6,8) {};
\node[fill=black,inner sep=0.25cm,outer sep=0pt,anchor=south west] at (6,9) {};
\node[fill=black,inner sep=0.25cm,outer sep=0pt,anchor=south west] at (7,0) {};
\node[fill=black,inner sep=0.25cm,outer sep=0pt,anchor=south west] at (7,1) {};
\node[fill=black,inner sep=0.25cm,outer sep=0pt,anchor=south west] at (7,3) {};
\node[fill=black,inner sep=0.25cm,outer sep=0pt,anchor=south west] at (7,4) {};
\node[fill=black,inner sep=0.25cm,outer sep=0pt,anchor=south west] at (7,5) {};
\node[fill=black,inner sep=0.25cm,outer sep=0pt,anchor=south west] at (7,7) {};
\node[fill=black,inner sep=0.25cm,outer sep=0pt,anchor=south west] at (7,8) {};
\node[fill=black,inner sep=0.25cm,outer sep=0pt,anchor=south west] at (7,9) {};
\node[fill=black,inner sep=0.25cm,outer sep=0pt,anchor=south west] at (11,0) {};
\node[fill=black,inner sep=0.25cm,outer sep=0pt,anchor=south west] at (11,1) {};
\node[fill=black,inner sep=0.25cm,outer sep=0pt,anchor=south west] at (11,3) {};
\node[fill=black,inner sep=0.25cm,outer sep=0pt,anchor=south west] at (11,4) {};
\node[fill=black,inner sep=0.25cm,outer sep=0pt,anchor=south west] at (11,5) {};
\node[fill=black,inner sep=0.25cm,outer sep=0pt,anchor=south west] at (11,7) {};
\node[fill=black,inner sep=0.25cm,outer sep=0pt,anchor=south west] at (11,8) {};
\node[fill=black,inner sep=0.25cm,outer sep=0pt,anchor=south west] at (11,9) {};

  \draw[step=1,help lines] (0,0) grid (12,10);

  \node[draw=red,ultra thick,inner sep=0.25cm,outer sep=0pt,anchor=south west] at (4,6) {};
  \node[draw=red,ultra thick,inner sep=0.25cm,outer sep=0pt,anchor=south west] at (5,6) {};
  \node[draw=red,ultra thick,inner sep=0.25cm,outer sep=0pt,anchor=south west] at (8,6) {};
  \node[draw=red,ultra thick,inner sep=0.25cm,outer sep=0pt,anchor=south west] at (5,3) {};
  \node[draw=red,ultra thick,inner sep=0.25cm,outer sep=0pt,anchor=south west] at (8,0) {};

    \draw[step=1,ultra thick, draw=blue] (4,0) to (4,7) to (9,7) to (9,0) to (4,0);

\end{tikzpicture}
\end{center}
\caption{Solution to problem from Figure \ref{DBGEx}.}\label{DBGExS}
\end{figure}
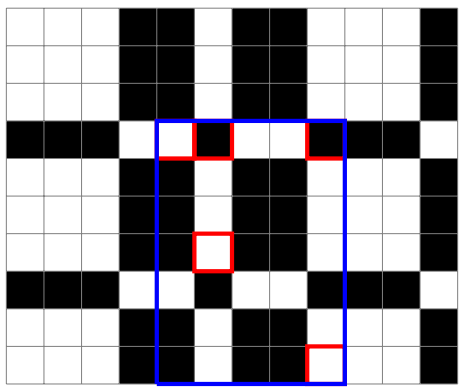

Finally, we provide a larger, non-binary example in Figure \ref{25x25}.  This figure is a de Bruijn torus for $d=5$ with $k=3$ and $n=2$ over the group $\mathcal{G} = \mathbb{Z} / 5 \mathbb{Z}$.  This torus will work to locate any of the following block patterns (amongst others) that contain four blocks.  Note that \textit{not} included is the two-by-two square, as this would give a connection graph that contains a cycle.

\begin{figure}
\centerline{\includegraphics[scale=0.6]{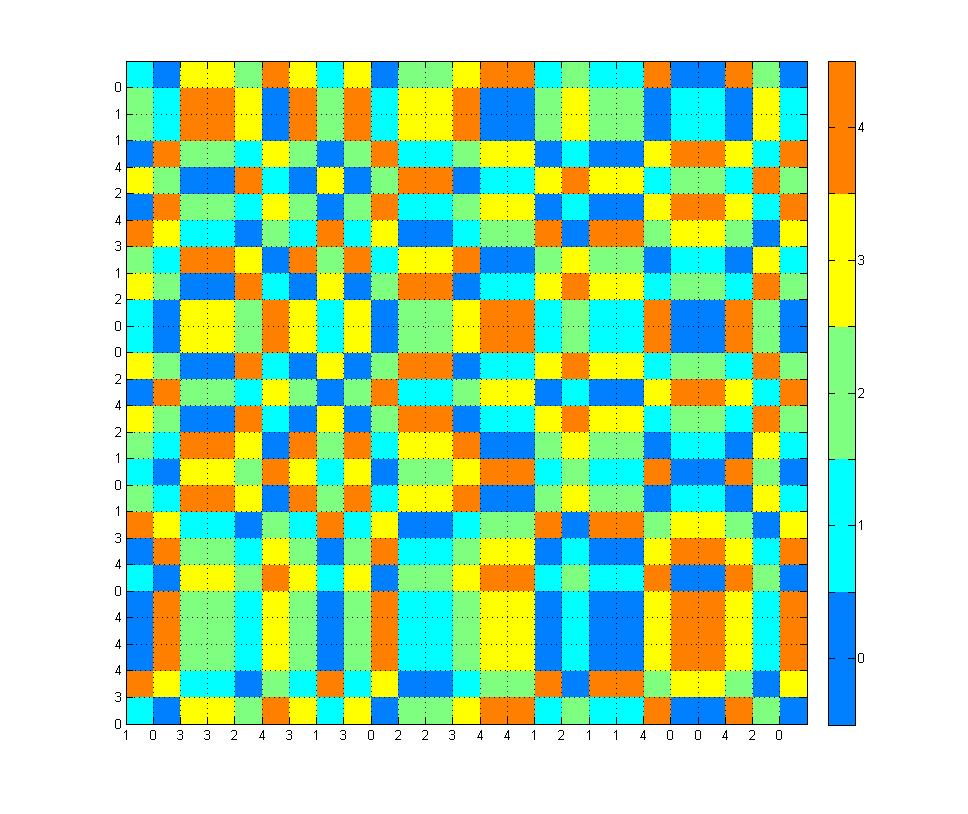}}

\begin{center}
\vspace{-10mm}
\begin{tikzpicture}[scale= 0.5 ]

  \draw[step=1,help lines] (0,3) grid (14,6);

  \draw[step=1,ultra thick, draw=red] (0,6) to (2,6) to (2,5) to (1,5) to (1,3) to (0,3) to (0,6);

  \draw[step=1,ultra thick, draw=red] (3,6) to (4,6) to (4,5) to (5,5) to (5,4) to (4,4) to (4,3) to (3,3) to (3,6);

  \draw[step=1,ultra thick, draw=red] (6,3) to (6,4) to (7,4) to (7,6) to (8,6) to (8,3) to (6,3);

  \draw[step=1,ultra thick, draw=red] (9,6) to (10,6) to (10,5) to (11,5) to (11,3) to (10,3) to (10,4) to (9,4) to (9,6);

  \draw[step=1,ultra thick, draw=red] (12,6) to (14,6) to (14,5) to (13,5) to (13,4) to (12,4) to (12,6);
  \draw[step=1,ultra thick, draw=red] (13,3) to (13,4) to (14,4) to (14,3) to (13,3);

\end{tikzpicture}
\end{center}

\caption{A de Bruijn torus for $d=5$, $k=3$, $n=2$, and $\mathcal{G}=\mathbb{Z}/5\mathbb{Z}$, accompanied by some valid block patterns.}\label{25x25}
\end{figure}

Let's consider an example on this figure.  Suppose that our block pattern is given, and we have filled it in as shown below.

\begin{center}
\begin{tikzpicture}[scale= 0.5 ]

  \draw[step=1,help lines] (3,3) grid (5,6);

  \draw[step=1,ultra thick, draw=red] (3,6) to (4,6) to (4,5) to (5,5) to (5,4) to (4,4) to (4,3) to (3,3) to (3,6);

  \node[anchor = south west] at (3,3) {0};
  \node[anchor = south west] at (3,4) {1};
  \node[anchor = south west] at (3,5) {0};
  \node[anchor = south west] at (4,4) {3};

  \node[anchor = south west] at (5,4) {or};

  \node[fill=blue,inner sep=0.25cm,outer sep=0pt,anchor=south west] at (6,3) {};
  \node[fill=blue,inner sep=0.25cm,outer sep=0pt,anchor=south west] at (6,5) {};
  \node[fill=cyan,inner sep=0.25cm,outer sep=0pt,anchor=south west] at (6,4) {};
  \node[fill=yellow,inner sep=0.25cm,outer sep=0pt,anchor=south west] at (7,4) {};

   \draw[step=1,help lines] (6,3) grid (8,6);
\end{tikzpicture}
\end{center}

If we assign the first column to have label $a$, then we get the horizontal projection is $(a, a+2)$ and the vertical projection is $(-a, -a+1, -a)$.  We begin with the vertical projection.  This could be any of the following (one for each choice of $a$):  $010, 121, 232, 343, 404$.  Note that only $404$ appears, and it appears in rows 5-7 from the bottom.  Thus we have $a=1$.  Next we find the horizontal projection $13$, which appears in columns 8-9 from the left.  Hence we find our entry in the submatrix shown below.

\begin{center}
\begin{tikzpicture}[scale= 0.5 ]

  \draw[step=1,help lines] (3,3) grid (5,6);

  \draw[step=1,ultra thick, draw=red] (3,6) to (4,6) to (4,5) to (5,5) to (5,4) to (4,4) to (4,3) to (3,3) to (3,6);

  \node[anchor = south west] at (3,3) {0};
  \node[anchor = south west] at (3,4) {1};
  \node[anchor = south west] at (3,5) {0};
  \node[anchor = south west] at (4,4) {3};
  \node[anchor = south west] at (4,5) {2};
  \node[anchor = south west] at (4,3) {2};

  \node[anchor = south west] at (5,4) {or};

  \node[fill=blue,inner sep=0.25cm,outer sep=0pt,anchor=south west] at (6,3) {};
  \node[fill=blue,inner sep=0.25cm,outer sep=0pt,anchor=south west] at (6,5) {};
  \node[fill=cyan,inner sep=0.25cm,outer sep=0pt,anchor=south west] at (6,4) {};
  \node[fill=yellow,inner sep=0.25cm,outer sep=0pt,anchor=south west] at (7,4) {};
  \node[fill=green,inner sep=0.25cm,outer sep=0pt,anchor=south west] at (7,5) {};
  \node[fill=green,inner sep=0.25cm,outer sep=0pt,anchor=south west] at (7,3) {};

   \draw[step=1,help lines] (6,3) grid (8,6);
\end{tikzpicture}
\end{center}

This decoding algorithm is summarized below.  Note that this algorithm provides us with an $\mathcal{O}(|\mathcal{Q}|+|\mathcal{S}|)$ search method, instead of the the standard $\mathcal{O}(|\mathcal{Q}|\times |\mathcal{S}|)$, as is needed for arbitrary grids as well as pseudorandom arrays.

\begin{algorithm}
\caption{Decoding Algorithm}
\begin{algorithmic}[1]
\Procedure{LocateBlocks}{$\mathcal{B}$}\Comment{Input:  Pattern to locate}
\State Determine vertical projection set
\State Search quotient string for vertical projection
\State Determine horizontal projection
\State Search de Bruijn sequence for horizontal projection
\EndProcedure

\Return Vertical location, horizontal location \Comment{row, column position}
\end{algorithmic}
\end{algorithm}

Note that Theorem \ref{tree} requires that our connection graph be a tree.  If it is not a tree, then it is either disconnected, or it contains a cycle.  We investigate each of these situations independently.

\begin{thm}
    Let $\mathcal{S}$ be a de Bruijn sequence for $d$-ary strings of length $n$ with a given comb pattern $P_1$.  Let $\mathcal{Q}$ be a quotient string with comb pattern $P_2$ for $d$-ary strings of length $k$.  Suppose that $\mathcal{G}$ is a group of order $d$.  Place $n+k-1$ blocks on the grid in a block pattern $\mathcal{B}$ so the connection graph for $\mathcal{B}$ is a forest made up of $c$ components, and so that the horizontal projection of $\mathcal{B}$ corresponds to $P_1$ and the vertical projection of $\mathcal{B}$ corresponds to $P_2$.  Then the grid $\mathcal{T}$ produced by $\mathcal{Q} \times \mathcal{S}$ using the group $\mathcal{G}$ contains every possible $n+k-1$ $d$-ary combination in block pattern $\mathcal{B}$ exactly $d^{c-1}$ times.
\end{thm}
\begin{proof}
    We first note that each component will have its own vertical (and corresponding horizontal) projection. Label the components as $C_1, C_2, \ldots , C_c$.  Using our column/row labeling procedure outlined in the proof of Theorem \ref{tree}, we will find column and row labels for each component in terms of a single variable (associated with whichever row/column started our algorithm).  Suppose that the labeling for $C_i$ is given in terms of $a_i$, and the vertical projection in terms of $a_i$ is given by $V_i(a_i)$ while the horizontal projection is given as $H_i(a_i)$.  That is, we get a projection labeling similar to the example shown below.  Note that this example is simplified so that the components are grouped clearly.  It could instead be the case that your projections are not consecutive rows/columns, but mixed.

    \begin{center}
    \begin{tikzpicture}[scale= 0.5 ]

    \draw (8,0) rectangle (10,2);
    \draw (5,3) rectangle (7,5);
    \draw (3,5) rectangle (5,7);

    \node[] at (4,6) {1};
    \node[] at (6,4) {2};
    \node[] at (9,1) {$c$};
    \node[anchor = south west] at (7,2) {$\ddots$};

    \draw[ultra thick] (3,0) rectangle (10,7);

    \draw (1,0) rectangle (2,2);
    \node[anchor= west] at (1,1) {$V_c$};
    \node[anchor = south west] at (1,2) {$\vdots$};
    \draw (1,3) rectangle (2,5);
    \node[anchor= west] at (1,4) {$V_2$};
    \draw (1,5) rectangle (2,7);
    \node[anchor= west] at (1,6) {$V_1$};

    \draw (3,8) rectangle (5,9);
    \node[anchor=south] at (4,8) {$H_1$};
    \draw (5,8) rectangle (7,9);
    \node[anchor=south] at (6,8) {$H_2$};
    \node[anchor=south west] at (7,8) {$\cdots$};
    \draw (8,8) rectangle (10,9);
    \node[anchor=south] at (9,8) {$H_c$};
    \end{tikzpicture}
    \end{center}

    As stated previously, each of these projections is dependent on the input variable $a_i \in [d]$.  We must find a $k$-tuple in the quotient de Bruijn string such that each $V_i(a_i)$ has an equivalence class representative in the correct location.  Considering all possible combinations of all possible equivalence class representatives, there are $d^c$ different possible strings that we must look for in our quotient string $\mathcal{Q}$.  However, since only one from each equivalence class appears in $\mathcal{Q}$, we will able to locate exactly $d^{c-1}$ of them in $\mathcal{Q}$, and these will be total vertical projections for the connection graph.  Each one of these total vertical projections appearing in $\mathcal{Q}$ will produce a corresponding total horizontal projection.  Hence the pattern we are attempting to locate will in fact appear $d^{c-1}$ times in our grid.
\end{proof}

\begin{thm}
        Let $\mathcal{S}$ be a de Bruijn sequence for $d$-ary strings of length $n$ with a given comb pattern $P_1$.  Let $\mathcal{Q}$ be a quotient string with comb pattern $P_2$ for $d$-ary strings of length $k$.  Suppose that $\mathcal{G}$ is a group of order $d$.  Place $n+k-1$ blocks on the grid in a block pattern $\mathcal{B}$ so the connection graph for $\mathcal{B}$ contains a cycle, and so that the horizontal projection of $\mathcal{B}$ corresponds to $P_1$ and the vertical projection of $\mathcal{B}$ corresponds to $P_2$.  Then there exists an $n+k-1$ $d$-ary combination in block pattern $\mathcal{B}$ that cannot be found within the grid $\mathcal{T}$ produced by $\mathcal{Q} \times \mathcal{S}$ using group $\mathcal{G}$.
\end{thm}
\begin{proof}
    Suppose that our connection graph contains a cycle.  Then there exists a block $B_i$ in the connection graph that has two neighbors preceding it in the ordering that was determined.  Call these neighbors $B_s$ and $B_t$.  When we determine the row and column labels for our blocks, depending on some variable $a$, the blocks $B_s$ and $B_t$ provide a row label $r(a)$ and a column label $c(a)$ for $B_i$.  When we are given a filled-in block pattern to locate in the grid, we will only be able to find block patterns in which the entry for $B_i$ is equal to $r(a)+c(a)$.  That is, if $B_i = b$, we must have $b = r(a)+c(a)$.  Note that this will completely determine our value for $a$, and so our vertical projection will be specific and cannot be replaced with a different representative from the same equivalence class. As exactly one representative from each class appears in $\mathcal{Q}$, this implies that for exactly one choice of $b \in [d]$ we can find the filled-in block pattern, but for the $d-1$ other choices we cannot.
\end{proof}

\section{Variations on the De Bruijn Torus Problem}\label{ucycles}

We now consider using other combinatorial objects rather than $d$-ary strings.  This requires us to jump from the world of de Bruijn sequences to the land of universal cycles.  To start, we generalize our definition of quotient strings to consider objects other than $d$-ary strings.

\begin{defn}
    Let $\mathcal{C}$ be a set of combinatorial objects, and suppose that we have an equivalence relation defined over the set that provides a partition of $\mathcal{C}$ into parts $P_1, P_2, \ldots, P_t$.  A \textbf{quotient string} for $\mathcal{C}$ is a string that contains exactly one consecutive substring from each part exactly once.  It is essentially a universal cycle for a set of equivalence class representatives.
\end{defn}

We now consider $k$-permutations.  For this we introduce a few new definitions.

\begin{defn}
    A \textbf{difference pattern} for a $k$-permutation of $[n]$ given by $x_1x_2 \ldots x_k$ is the string $d_1d_2 \ldots d_{k-1}$ where $d_i = x_{i+1}-x_i$.  A \textbf{difference sequence} for $k$-permutations of $[n]$ is a universal string over all possible difference patterns for $k$-permutations of $[n]$.  
\end{defn}

The difference sequence for $k$-permutation is the quotient string that we will use to produce grids.  A necessary requirement for this is a proof that such difference sequences exist.

\begin{lem}
    There exists a difference sequence for $k$-permutations of $[n]$ for all $n,k \in \mathbb{Z}^+$ with $3 \leq k < n$.
\end{lem}
\begin{proof}
    First, we will construct the transition digraph, $\mathcal{G}_d$, for difference patterns for $k$-permutations of $[n]$, with the following vertex and edge sets.
    
    \textit{Vertices:}  $d_1d_2 \ldots d_{k-2}$ where $d_1d_2 \ldots d_{k-2}d_{k-1}$ is a difference pattern for a $k$-permutation of $[n]$ for some $d_{k-1}$.

    \textit{Edges:}  $d_1d_2d_3 \ldots d_{k-2} \rightarrow d_2d_3 \ldots d_{k-2}d_{k-1}$, where $d_1d_2d_3 \ldots d_{k-2}d_{k-1}$ is a valid difference pattern for a $k$-permutation of $[n]$.
    
    As is standard practice in the universal cycle literature, we will show that this graph is eulerian by illustrating that it is both balanced and weakly connected.  Once this graph is known to be eulerian, we can find a difference sequence for $k$-permutations of $[n]$ simply by following any Euler tour in $\mathcal{G}_d$.  To prove that $\mathcal{G}_d$ is eulerian, we construct a second (separate but related) digraph, $\mathcal{G}_p$.  This graph $\mathcal{G}_p$ is the transition digraph for $k$-permutations of $[n]$ is the graph with the following vertex and edge sets.

    \textit{Vertices:}  $x_1x_2 \ldots x_{k-1}$ where $x_1x_2 \ldots x_{k-1}x_k$ is a $k$-permutation of $[n]$ for some $x_k$.

    \textit{Edges:}  $x_1x_2x_3 \ldots x_{k-1} \rightarrow x_2 x_3 \ldots x_{k-1}x_k$ where $x_1 x_2 x_3 \ldots x_{k-1}x_k$ is a $k$-permutation of $[n]$.
    
    Now that we have both digraphs defined, we show a special relationship between $\mathcal{G}_d$ and $\mathcal{G}_p$ by looking at the mapping $\varphi : \mathcal{G}_p \mapsto \mathcal{G}_d$, which maps each $k$-permutation's prefix to the prefix of its difference pattern.  This mapping is equivalent to contracting each difference pattern's equivalence class in $\mathcal{G}_p$ to produce $\mathcal{G}_d$.  To prove that this statement is true, we need only show that there is an edge from $X^- = x_1x_2 \ldots x_{k-1}$ to $X^+ = x_2x_3 \ldots x_{k-1}x_k$ if and only if there is an edge from $D_{X^-}$ to $D_{X^+}$, the corresponding difference sequences.

    First, we note that $D_{X^-} = d_1d_2 \ldots d_{k-2}$ where $d_i= x_{i+1}-x_i$ for $i \in [k-2]$, and $D_{X^+}=\delta_2 \delta_3 \ldots \delta_{k-1}$ where $\delta_i = x_{i+1}-x_i$ for $i \in \{2,3, \ldots , k-1\}$.  Hence $d_2d_3 \ldots d_{k-2} = \delta_2 \delta_3 \ldots \delta_{k-2}$, and so we have $D_{X^-} \rightarrow D_{X^+}$ in $\mathcal{G}_d$.

    For the reverse, suppose that we have an edge $d_1d_2 \ldots d_{k-2} \rightarrow d_2d_3 \ldots d_{k-1}$ in $\mathcal{G}_d$.  Then the difference pattern $D=d_1d_2 \ldots d_{k-1}$ corresponds to a class of $k$-permutations of $[n]$ of the form $\left(x, x+d_1, x+d_1+d_2, \ldots , x+ \sum_{i=1}^{k-1} d_i \right)$ (for any $x \in [n]$).  These $k$-permutations correspond to the following edges in $\mathcal{G}_p$. $$\left(x, x+d_1 , x+d_1+d_2, \ldots , x+ \sum_{i=1}^{k-2} d_i \right)$$ $$ \rightarrow \left(x+d_1, x+d_1+d_2, \ldots ,x+ \sum_{i=1}^{k-2}d_i, x+\sum_{i=1}^{k-1} \right)$$  Thus our mapping $\varphi$ performs as stated.  In other words, $\varphi$ maps classes of $n$ vertices in $\mathcal{G}_p$ to one representative in $\mathcal{G}_d$, and maps classes of $n$ edges in $\mathcal{G}_p$ to one edge in $\mathcal{G}_d$.
    
    Next, from \cite{kPerms}, we know that $\mathcal{G}_p$ is eulerian.  Now we can use our mapping $\varphi$ to prove that $\mathcal{G}_d$ is eulerian too, as the $k$-permutation digraph is connected if and only if $\mathcal{G}_d$ is connected.  Lastly, because of our mapping $\varphi$, it is clear that the $\mathcal{G}_d$ is balanced if and only if $\mathcal{G}_p$ is balanced.  Hence the $\mathcal{G}_d$ is eulerian.
\end{proof}

Now that we know the correct quotient string exists, we use a construction similar to that in Theorem \ref{tree} to produce our main result for this section.

\begin{thm}
    Let $\mathcal{S}$ be a universal cycle for $k$-permutations of $[n]$ and let $\mathcal{D}$ be a difference string for $\ell$-permutations of $[n]$.  Let $\mathcal{G}$ be a group of order $n$, with elements labelled by $[n]$ and operation $\oplus$.  Construct the torus for $\mathcal{D} \times \mathcal{S}$ over $\mathcal{G}$ and call it $\mathcal{T}$.  Fix a block pattern $\mathcal{B}$ such that we have $k$ blocks horizontally and $\ell$ blocks vertically.  Fill block pattern $\mathcal{B}$ with elements from $\mathcal{G}$ arbitrarily such that the $k$ horizontal blocks form a $k$-permutation and the $\ell$ vertical blocks form an $\ell$-permutation.  Then this filled block pattern must appear exactly once in $\mathcal{T}$.
\end{thm}

We now provide an example and consider the following.  If we want to produce a torus for $2$-permutations of $[4]$, we need a universal cycle for this set (for the horizontal) and a difference string as well (for the vertical).  In this example, we will use the Klein-4 group for our group $\mathcal{G}$, with addition given by the following group table.

$$\begin{array}{c|cccc}
    \oplus & 0 & 1 & 2 & 3 \\
    \hline
    0 & 0 & 1 & 2 & 3 \\
    1 & 1 & 0 & 3 & 2 \\
    2 & 2 & 3 & 0 & 1 \\
    3 & 3 & 2 & 1 & 0
\end{array}$$

Our universal cycle is given by $013120230321$, and our quotient string is given by $013$.  Then our $2 \times 2$ torus is the following $3 \times 12$ grid.

$$\begin{array}{c|cccccccccccc}
    \oplus & 0 & 1 & 3 & 1 & 2 & 0 & 2 & 3 & 0 & 3 & 2 & 1 \\
    \hline
    0 & 0 & 1 & 3 & 1 & 2 & 0 & 2 & 3 & 0 & 3 & 2 & 1 \\
    1 & 1 & 0 & 2 & 0 & 3 & 1 & 3 & 2 & 1 & 2 & 3 & 0 \\
    3 & 3 & 2 & 0 & 2 & 1 & 3 & 1 & 0 & 3 & 0 & 1 & 2
\end{array}$$

For example, in the example for 2-permutations of $[4]$, consider our grid constructed previously and the marked block pattern.

\begin{center}
\begin{tikzpicture}[scale= 0.5 ]

\node[anchor=south west] at (0,0) {3};
\node[anchor=south west] at (1,0) {2};
\node[anchor=south west] at (2,0) {0};
\node[anchor=south west] at (3,0) {2};
\node[anchor=south west] at (4,0) {1};
\node[anchor=south west] at (5,0) {3};
\node[anchor=south west] at (6,0) {1};
\node[anchor=south west] at (7,0) {0};
\node[anchor=south west] at (8,0) {3};
\node[anchor=south west] at (9,0) {0};
\node[anchor=south west] at (10,0) {1};
\node[anchor=south west] at (11,0) {2};

\node[anchor=south west] at (0,1) {1};
\node[anchor=south west] at (1,1) {0};
\node[anchor=south west] at (2,1) {2};
\node[anchor=south west] at (3,1) {0};
\node[anchor=south west] at (4,1) {3};
\node[anchor=south west] at (5,1) {1};
\node[anchor=south west] at (6,1) {3};
\node[anchor=south west] at (7,1) {2};
\node[anchor=south west] at (8,1) {1};
\node[anchor=south west] at (9,1) {2};
\node[anchor=south west] at (10,1) {3};
\node[anchor=south west] at (11,1) {0};

\node[anchor=south west] at (0,2) {0};
\node[anchor=south west] at (1,2) {1};
\node[anchor=south west] at (2,2) {3};
\node[anchor=south west] at (3,2) {1};
\node[anchor=south west] at (4,2) {2};
\node[anchor=south west] at (5,2) {0};
\node[anchor=south west] at (6,2) {2};
\node[anchor=south west] at (7,2) {3};
\node[anchor=south west] at (8,2) {0};
\node[anchor=south west] at (9,2) {3};
\node[anchor=south west] at (10,2) {2};
\node[anchor=south west] at (11,2) {1};

  \draw[step=1,help lines] (0,0) grid (12,3);
  \draw[step=1,ultra thick] (0,3) to (2,3) to (2,1) to (1,1) to (1,2) to (0,2) to (0,3);

\end{tikzpicture}
\end{center}

\section{Future Work}\label{FutureWork}

There are many directions for future research to consider.  One of the most obvious concerns our original motivation from Problem 480 \cite{ResearchProbs}:  how do we modify our methods to allow for cycles in the connection graph?  The original task is to use a block pattern that is simply a rectangle, which potentially contains many cycles.

Corresponding to our variations that utilize universal cycles, there are many open problems simply by consider the various combinatorial objects that have been `ucycled'.  For example, these could consider subsets, partitions, weak orders, etc.  The vast literature on ucycles provides plenty of opportunities for future work.

Finally, when we consider the real-world applications we must allow for things like sensor failure.  Is there any way to build redundancy into our tori so that if utilized for robotic vision (self-detection for robots on the grid), can our methods handle the failure of one sensor (i.e. losing one block in our block pattern)?  Alternatively, can we make this method robust to rotations?  In other words, if the robot rotates $90^\circ$, is it still able to self-locate?

As we consider potential applications and variations, there is a plethora of research possibilities available.
\bibliographystyle{amsplain}

\end{document}